\documentclass[11pt,leqno]{article}

\usepackage[active]{srcltx}

\usepackage{amsfonts,latexsym,amsmath,amssymb,amsthm}
\usepackage{fullpage}
\newtheorem{theorem}{Theorem}[section]
\newtheorem{lemma}[theorem]{Lemma}

\newtheorem{proposition}[theorem]{Proposition}

\newtheorem{remark}[theorem]{Remark}

\theoremstyle{definition}

\theoremstyle{remark}

\newtheorem*{note*}{Note}

\numberwithin{equation}{section}

\makeatletter
\newcommand{\rank}{\mathop{\operator@font rank}}
\newcommand{\conv}{\mathop{\operator@font conv}}
\newcommand{\vol}{\mathop{\operator@font vol}}
\newcommand{\onetagright}{\tagsleft@false}
\makeatother

\newcommand{\ls}{\leqslant}
\newcommand{\gr}{\geqslant}
\renewcommand{\epsilon}{\varepsilon}

\usepackage{color}

\usepackage{hyperref}

\begin{document}
\small

\title{\bf Brascamp-Lieb inequality and quantitative versions of Helly's theorem}

\medskip

\author{Silouanos Brazitikos}

\date{}

\maketitle

\begin{abstract}
\footnotesize We provide a number of new quantitative versions of Helly's theorem. For example,
we show that for every family $\{P_i:i\in I\}$ of closed half-spaces
\begin{equation*}P_i=\{ x\in {\mathbb R}^n:\langle x,w_i\rangle \ls 1\}\end{equation*}
in ${\mathbb R}^n$ such that $P=\bigcap_{i\in I}P_i$ has positive volume, there exist $s\ls \alpha n$ and
$i_1,\ldots , i_s\in I$ such that
\begin{equation*}|P_{i_1}\cap\cdots\cap P_{i_s}|\ls (Cn)^n\,|P|,\end{equation*}
where $\alpha , C>0$ are absolute constants. These results complement and improve
previous work of B\'{a}r\'{a}ny-Katchalski-Pach and Nasz\'{o}di. Our method combines the work of Srivastava on approximate
John's decompositions with few vectors, a new estimate on the corresponding constant
in the Brascamp-Lieb inequality and an appropriate variant of Ball's proof of the reverse isoperimetric
inequality.
\end{abstract}

\section{Introduction}

Our starting point is a quantitative version of Helly's theorem on convex sets in Euclidean space. Helly's theorem states that
if ${\mathcal P}=\{ P_i:i\in I\}$ is a finite family of at least $n+1$ convex sets in ${\mathbb R}^n$ and if any $n+1$ members of ${\mathcal P}$ have
non-empty intersection then $\bigcap_{i\in I}P_i\neq\emptyset $. B\'{a}r\'{a}ny, Katchalski and Pach proved in \cite{BKP-1982}
(see also \cite{BKP-1984}) the following quantitative ``volume version":

\smallskip

{\sl Let ${\mathcal P}=\{ P_i:i\in I\}$ be a finite family of convex sets in ${\mathbb R}^n$. If the intersection of
any $2n$ or fewer members of ${\mathcal P}$ has volume greater than or equal to $1$, then
$\left |\bigcap_{i\in I}P_i\right |\gr c_n$, where $c_n>0$ is a constant depending only on $n$.}

\smallskip
Using the fact that every (closed) convex set is the intersection of a family of closed half-spaces and a simple compactness argument
(see \cite{BKP-1982}) one can remove the restriction that ${\mathcal P}$ is finite and also assume that each $P_i$ is a closed half-space.
Therefore, the theorem of B\'{a}r\'{a}ny, Katchalski and Pach is equivalently stated as follows:

\smallskip

{\sl Let ${\mathcal P}=\{ P_i:i\in I\}$ be a family of closed half-spaces in ${\mathbb R}^n$
such that $\left |\bigcap_{i\in I}P_i\right |>0$. There exist $s\ls 2n$
and $i_1,\ldots ,i_s\in I$ such that
\begin{equation}\left |P_{i_1}\cap \cdots \cap P_{i_s}\right |\ls C_n\left |\bigcap_{i\in I} P_i\right |,\end{equation}
where $C_n>0$ is a constant depending only on $n$.}

\smallskip

Note that the cube $[-1,1]^n$ in ${\mathbb R}^n$ can be written as the intersection of the $2n$ closed half-spaces $H_j^{\pm }:=\{ x:\langle x,\pm e_j\rangle \ls 1\}$
and that the intersection of any $2n-1$ of these half-spaces has infinite volume; this shows that one cannot replace $2n$ by $2n-1$ in the statement above.
In \cite{BKP-1982} the authors offered a bound $C_n\ls n^{2n^2}$ for the constant $C_n$ and they conjectured that one might actually
have $C_n\ls n^{cn}$ for an absolute constant $c>0$. Nasz\'{o}di \cite{Naszodi-2015} has recently verified this conjecture; namely, he proved
a volume version of Helly's theorem with $C_n\ls (Cn)^{2n}$, where $C>0$ is an absolute constant. In Section 3 we present a slight modification of
Nasz\'{o}di's argument which leads to the exponent $\frac{3n}{2}$ instead of $2n$:

\begin{theorem}\label{th:intro-1}Let ${\mathcal P}=\{ P_i:i\in I\}$ be a family of closed half-spaces such that $\left |\bigcap_{i\in I}P_i\right |>0$.
We may find $s\ls 2n$ and $i_1,\ldots ,i_s\in I$ such that
\begin{equation}\label{eq:nasz-0-intro}|P_{i_1}\cap\cdots\cap P_{i_s}|\ls (Cn)^{\frac{3n}{2}}\left |\bigcap_{i\in I}P_i\right |,\end{equation}
where $C>0$ is an absolute constant.
\end{theorem}

The aim of this work is to study a natural question that arises from Theorem \ref{th:intro-1}. Given $N>2n$ we would like to estimate the quantity
\begin{equation}C_{n,N}=\sup\frac{|P_{i_1}\cap\cdots\cap P_{i_N}|}{\left |\bigcap_{i\in I}P_i\right |}\end{equation}
where the supremum is over all families ${\mathcal P}=\{ P_i:i\in I\}$ of closed half-spaces with $\left |\bigcap_{i\in I}P_i\right |>0$.
We would also like to study the same question in the case of families of symmetric strips in ${\mathbb R}^n$.

Starting with the symmetric case, our main result is the next theorem.

\begin{theorem}\label{th:strips-intro}Let $\{P_i:i\in I\}$ be a family of symmetric strips
\begin{equation}P_i=\{ x\in {\mathbb R}^n:|\langle x,w_i\rangle |\ls 1\}\end{equation}
in ${\mathbb R}^n$, such that $P=\bigcap_{i\in I}P_i$ has positive volume. For every $d>1$ there exist $s\ls dn$ and $i_1,\ldots ,i_s\in I$
such that
\begin{equation}|P_{i_1}\cap\cdots\cap P_{i_s}|\ls \left (\frac{4\gamma_d}{\pi}\right)^{\frac{n}{2}}\Gamma \left(\frac{n}{2}+1\right )|P|,\end{equation}
where $\gamma_d:=\left(\frac{\sqrt{d}+1}{\sqrt{d}-1}\right)^2$.\end{theorem}

Note that if $d\gg 1$ then the constant $C_{n,\lfloor dn\rfloor }$ is bounded by $(Cn)^{\frac{n}{2}}$. In the non-symmetric case we first use a similar
strategy (whose details are of course more delicate) to obtain an estimate
comparable to the one in Theorem \ref{th:intro-1}.

\begin{theorem}\label{th:halfspaces-intro1}Let $\{P_i:i\in I\}$ be a family of closed half-spaces
\begin{equation}P_i=\{ x\in {\mathbb R}^n:\langle x,v_i\rangle \ls 1\}\end{equation}
in ${\mathbb R}^n$, such that $P=\bigcap_{i\in I}P_i$ has positive volume. For every $d>1$ there exist $s\ls (d+1)(n+1)$ and
$i_1,\ldots , i_s\in I$ such that
\begin{equation}|P_{i_1}\cap\cdots\cap P_{i_s}|\ls \gamma_d^{\frac{n+1}{2}}\frac{n^{n/2}(n+1)^{3(n+1)/2}}{\pi^{\frac{n}{2}}n!}\Gamma \left(\frac{n}{2}+1\right )\,|P|
\ls \gamma_d^{\frac{n+1}{2}}(Cn)^{\frac{3n}{2}}|P|,\end{equation}where $C>0$ is an absolute constant.
\end{theorem}

Note that Theorem \ref{th:halfspaces-intro1} gives the same dependence on $n$ as Theorem \ref{th:intro-1}. In fact, Theorem \ref{th:intro-1} is stronger
if what matters is to use (the smallest possible number of) $2n$ of the half-spaces $P_i$. On the other hand, there is a (small) difference in the value
of the constant $C$ involved in the two statements: the proof of Theorem \ref{th:intro-1} works with $C=2\sqrt[3]{\pi }$, while the proof of Theorem \ref{th:halfspaces-intro1} works
with $C_d=\left(\frac{e\gamma_d}{2\pi }\right )^{\frac{1}{3}}$ (which is smaller than $C$ if $d$ is large enough).

However, if we relax the condition on the number $s$ of half-spaces that we use (but still require that it is proportional to the
dimension) we are able to (significantly) improve the exponent in the constant $C_{n,N}$ from $\frac{3n}{2}$ to $n$:

\begin{theorem}\label{th:halfspaces-intro2}There exists an absolute constant $\alpha >1$ with the following
property: for every family $\{P_i:i\in I\}$ of closed half-spaces
\begin{equation}P_i=\{ x\in {\mathbb R}^n:\langle x,v_i\rangle \ls 1\}\end{equation}
in ${\mathbb R}^n$, such that $P=\bigcap_{i\in I}P_i$ has positive volume, there exist $s\ls \alpha n$ and
$i_1,\ldots , i_s\in I$ such that
\begin{equation}|P_{i_1}\cap\cdots\cap P_{i_s}|\ls (Cn)^n\,|P|,\end{equation}
where $C>0$ is an absolute constant.
\end{theorem}

Let us note that, in the recent paper \cite{LHRS-2015}, De Loera, La Haye, Rolnick and Sober\'{o}n have presented many interesting
results, both continuous and discrete, that may be viewed as quantitative versions of Carath\'{e}odory's, Helly's and Tverberg's theorems.
For example, they prove that for every $n\gr 1$ and $\varepsilon >0$ there exists a positive integer $N(n,\varepsilon )$ with the following
property: if ${\mathcal F}=\{ F_i:i\in I\}$ is a finite family of convex sets in ${\mathbb R}^n$
such that $\left |F_{i_1}\cap\cdot\cap F_{i_s}\right |\gr 1$ for all $s\ls Nn$ and all $i_1,\ldots ,i_s\in I$, then
\begin{equation}\left |\bigcap_{i\in I}F_i\right |\gr\frac{1}{1+\varepsilon }.\end{equation}
They also obtain a variant of this statement in which volume is replaced by diameter, as well as a ``colorful" volume version
of Helly's theorem. We would like to emphasize that the ``philosophy" of all these results is completely different from the one
in our work. The parameter $N(n,\varepsilon )$ is defined as the smallest integer such that, for every convex set $K\subset {\mathbb R}^n$
of positive volume there exists a polytope $P\supseteq K$ with at most $M(n,\varepsilon )$ facets such that
$|P|\ls (1+\varepsilon )|K|$, and it is known that $N(n,\varepsilon )$ is exponential in $n$ and $\varepsilon $: one has
\begin{equation}\left (\frac{c_1n}{\varepsilon }\right )^{\frac{n-1}{2}}\ls N(n,\varepsilon )\ls\left(\frac{c_2n}{\varepsilon }\right)^{\frac{n-1}{2}}.\end{equation}
We are interested in the best lower bound that one can obtain for $\left |\bigcap_{i\in I}F_i\right |$ in terms of a lower bound
for the volume of the intersection of any $N\sim n$ of the sets $F_i$ ($N$ is assumed proportional to the dimension).

We close this introductory section by briefly explaining the main ideas behind the proof of our results in the non-symmetric case. We may assume
that $P=\bigcap_{i\in I}\{x\in {\mathbb R}^n:\langle x,v_i\rangle\ls 1\}$ has finite volume and, since the statements are affinely invariant,
that $P$ is in John's position, i.e. the ellipsoid of maximal volume inscribed in $P$ is the
Euclidean unit ball $B_2^n$. Then, we have John's decomposition of the identity (see Section 2 for background
information): there exists $J\subseteq I$ such that $v_j$, $j\in J$ are contact points of $P$ and $B_2^n$ and there are
positive scalars $a_j$, $j\in J$ such that
\begin{equation}I_n=\sum_{j\in J}a_jv_j\otimes v_j\quad\hbox{and}\quad \sum_{j\in J}a_jv_j=0.\end{equation}
Given $d>1$ we would like to extract a subset $\sigma $ of $J$, of cardinality $dn$, which still
forms an approximate John's decomposition of the identity with suitable weigths. To this end, for the proof of Theorem \ref{th:halfspaces-intro1} we use a result of Batson, Spielman and Srivastava from \cite{BSS-2009}: there exists a subset $\sigma\subseteq J$ with $|\sigma |\ls dn$ and $b_j>0$, $j\in \sigma $, such that
\begin{equation}I_n\preceq \sum_{j\in\sigma }b_ja_jv_j\otimes v_j \preceq \gamma_dI_n,\end{equation}
where $\gamma_d:=\left(\frac{\sqrt{d}+1}{\sqrt{d}-1}\right)^2$. For the proof of Theorem \ref{th:halfspaces-intro2}
we use a second, more delicate, theorem of Srivastava from \cite{Srivastava-2012} (see Section 4 for the precise statement).

Then, we would like to exploit an appropriate variant of Ball's proof of the reverse isoperimetric inequality in \cite{Ball-1991}
in order to estimate the volume of $Q:=\bigcap_{j\in\sigma }P_j$ using the Brascamp-Lieb inequality (see Section 5).
The main problem now is to obtain an estimate for the constant in the Brascamp-Lieb inequality that corresponds to an approximate
John's decomposition. To the best of our knowledge this question had not been studied. Our main technical result is
the next theorem; we feel that it is a useful tool of independent interest.

\begin{theorem}\label{BL-approx-intro}Let $\gamma >1$. Let $u_1,\ldots ,u_s \in S^{n-1}$ and $c_1,\ldots ,c_s>0$ satisfy
\begin{equation}I_n\preceq A:=\sum_{j=1}^sc_ju_j\otimes u_j \preceq \gamma I_n\end{equation}
and set $\kappa_j=c_j\langle A^{-1}u_j, u_j\rangle >0$, $1\ls j\ls s$. If
$f_1,\ldots ,f_s:{\mathbb R}\to {\mathbb R}^+$ are measurable
functions then
\begin{equation}\int_{{\mathbb R}^n}\prod_{j=1}^sf_j^{\kappa_j}(\langle x,u_j\rangle )dx
\ls \gamma^{\frac{n}{2}}\prod_{j=1}^s \left(\int_{{\mathbb R}} f_j(t)dt\right)^{\kappa_j}.\end{equation}
\end{theorem}

In Section 6 we present the proofs of the main results.

\section{Notation and background}

We work in ${\mathbb R}^n$, which is equipped with a Euclidean structure $\langle\cdot ,\cdot\rangle $. We denote by $\|\cdot \|_2$
the corresponding Euclidean norm, and write $B_2^n$ for the Euclidean unit ball and $S^{n-1}$ for the unit sphere.
Volume is denoted by $|\cdot |$. We write $\omega_n$ for the volume of $B_2^n$ and $\sigma $ for the rotationally invariant probability
measure on $S^{n-1}$. We will denote by $P_F$ the orthogonal projection from $\mathbb R^{n}$ onto $F$. We also define
$B_F=B_2^n\cap F$ and $S_F=S^{n-1}\cap F$.

The letters $c,c^{\prime }, c_1, c_2$ etc. denote absolute positive constants which may change from line to line. Whenever we write
$a\simeq b$, we mean that there exist absolute constants $c_1,c_2>0$ such that $c_1a\ls b\ls c_2a$.  Also, if $K,L\subseteq \mathbb R^n$
we will write $K\simeq L$ if there exist absolute constants $c_1,c_2>0$ such that $ c_{1}K\subseteq L \subseteq c_{2}K$.

We refer to the book of Schneider \cite{Schneider-book} for basic facts from the Brunn-Minkowski theory and to the book
of Artstein-Avidan, Giannopoulos and V. Milman \cite{AGA-book} for basic facts from asymptotic convex geometry.

\smallskip

A convex body in ${\mathbb R}^n$ is a compact convex subset $K$ of ${\mathbb R}^n$ with non-empty interior. We say that $K$ is
symmetric if $x\in K$ implies that $-x\in K$, and that $K$ is centered if its barycenter
\begin{equation}{\rm bar}(K)=\frac{1}{|K|}\int_Kx\,dx \end{equation} is at the origin. The polar body
$K^{\circ }$ of $K$ is defined by
\begin{equation}K^{\circ }:=\{ y\in {\mathbb R}^n: \langle x,y\rangle \ls 1
\;\hbox{for all}\; x\in K\}. \end{equation}
The Blaschke-Santal\'{o} inequality states that for every centered convex body $K$ in ${\mathbb R}^n$ one has $|K||K^{\circ }|\ls\omega_n^2$,
with equality if and only if $K$ is an ellipsoid. The reverse Santal\'{o} inequality of Bourgain and V. Milman
\cite{Bourgain-VMilman-1987} states that there exists an absolute constant $c>0$ such that
\begin{equation}\left (|K||K^{\circ }|\right )^{1/n}\gr c/n,\end{equation}
where $c>0$ is an absolute constant, for every convex body $K$ in ${\mathbb R}^n$
which contains $0$ in its interior.

\smallskip

We say that a convex body $K$ is in John's position if the ellipsoid of maximal volume inscribed in $K$ is
the Euclidean unit ball $B_2^n$. John's theorem \cite{John-1948} states that
$K$ is in John's position if and only if $B_2^n\subseteq K$ and there exist $u_1,\ldots ,u_m\in {\rm
bd}(K)\cap S^{n-1}$ (contact points of $K$ and $B_2^n$) and positive real numbers $c_1,\ldots ,c_m$ such that
\begin{equation}\label{eq:bar-0}\sum_{j=1}^mc_ju_j=0\end{equation}
and the identity operator $I_n$ is decomposed in the form
\begin{equation}\label{eq:decomposition}I_n=\sum_{j=1}^mc_ju_j\otimes u_j, \end{equation}
where $(u_j\otimes u_j)(y)=\langle u_j,y\rangle u_j$. In the case where $K$ is symmetric, the second
condition \eqref{eq:decomposition} is enough (for any contact point $u$ we have that $-u$ is also
a contact point, and hence, having \eqref{eq:decomposition} we may easily produce a decomposition for which
\eqref{eq:bar-0} is also satisfied). In analogy to John's position, we say that a convex body $K$ is in L\"{o}wner's
position if the ellipsoid of minimal volume containing $K$ is the Euclidean unit ball
$B_2^n$. One can check that this holds true if and only if $K^{\circ }$ is in John's position;
 in particular, we have a decomposition of the identity similar to
\eqref{eq:decomposition}.

Assume that $u_1,\ldots ,u_m$ are unit vectors that satisfy John's decomposition \eqref{eq:decomposition} with some
positive weights $c_j$. Then, one has the useful identities
\begin{equation}\label{eq:trace}\sum_{j=1}^mc_j={\rm tr}(I_n)=n\quad \hbox{and}\quad
\sum_{j=1}^mc_j\langle u_j,z\rangle^2=1 \end{equation}
for all $z\in S^{n-1}$. Moreover,
\begin{equation}\label{eq:ball-inside}{\rm conv}\{u_1,\ldots ,u_m\}\supseteq\frac{1}{n}B_2^n.\end{equation}
In the symmetric case we actually have
\begin{equation}\label{eq:ball-inside-sym}{\rm conv}\{\pm u_1,\ldots ,\pm u_m\}\supseteq\frac{1}{\sqrt{n}}B_2^n.\end{equation}
Another useful fact, which goes back to the classical article of Dvoretzky and Rogers \cite{Dvoretzky-Rogers-1950},
is that we may choose $v_1,\ldots ,v_n$, among the $u_i$'s, which satisfy
\begin{equation}\label{eq:1.1}{\rm dist}(v_k,{\rm span}(v_1,v_2,\ldots ,v_{k-1}))\gr\sqrt{\frac{n-k+1}{n}}\end{equation}
for all $k=2,\ldots ,n$.

Finally, we state as a lemma a useful fact from linear algebra that will be used in Section 5.

\begin{lemma}\label{mdetlemma}
Let $A$ be an $n\times n$ invertible matrix. For any $u,v\in \mathbb{R}^n$ we have
\begin{equation}
\det(A+u\otimes v)=\det(A)(1+\langle A^{-1}u,v\rangle ).
\end{equation}
\end{lemma}

\begin{proof}Let $u,v\in {\mathbb R}^n$. Starting with the identity
\begin{equation}\begin{pmatrix} I_n & 0\\ v & 1\end{pmatrix}\begin{pmatrix} I_n+u\otimes v & u\\ 0 & 1\end{pmatrix}\begin{pmatrix} I_n & 0\\ -v & 1\end{pmatrix}
=\begin{pmatrix} I_n & u\\ 0 & 1+\langle u,v\rangle \end{pmatrix}\end{equation}
and taking determinants we see that $\det (I+u\otimes v)=1+\langle u,v\rangle $, which is the assertion of the lemma in the case $A=I_n$. Given any
$n\times n$ invertible matrix $A$ we write
\begin{equation}A+u\otimes v =A(I_n+A^{-1}(u\otimes v))=A(I_n+(A^{-1}u\otimes v)),\end{equation}
and applying the previous special case we obtain
\begin{equation}\det (A+u\otimes v)=\det (A)\det (I_n+(A^{-1}u\otimes v))=\det (A)(1+\langle A^{-1}u,v\rangle )\end{equation}
as claimed. \end{proof}

\section{A refinement of Nasz\'{o}di's argument}

We start with a refinement of Nasz\'{o}di's argument from \cite{Naszodi-2015}; our only new ingredient is the fact that every convex body $K$ contains a centrally symmetric
convex body $K_1$ of volume $|K_1|\gr 2^{-n}|K|$. Incorporating this in the original proof we obtain a better estimate.

\begin{theorem}\label{th:nasz-refined}Let ${\mathcal P}=\{ P_i:i\in I\}$ be a family of closed half-spaces such that $\left |\bigcap_{i\in I}P_i\right |>0$.
We may find $s\ls 2n$ and $i_1,\ldots ,i_s\in I$ such that
\begin{equation}\label{eq:nasz-0}|P_{i_1}\cap\cdots\cap P_{i_s}|\ls (Cn)^{\frac{3n}{2}}\left |\bigcap_{i\in I}P_i\right |,\end{equation}
where $C>0$ is an absolute constant.
\end{theorem}

\begin{proof} We start with a family ${\mathcal P}=\{ P_i:i\in I\}$ of closed half-spaces $P_i=\{x:\langle x,u_i\rangle\ls 1\}$
such that $\left |\bigcap_{i\in I}P_i\right |<\infty $. We may assume
that ${\mathcal P}$ is a finite family, therefore $P=\bigcap_{i\in I}P_i$ is a polytope. By affine invariance, we may also assume that
$P$ is in John's position. From John's theorem  there exists $J\subseteq I$ such that $u_j$, $j\in J$ are contact points of $P$
and $B_2^n$, and $a_j>0$, $j\in J$ such that
\begin{equation}\label{eq:nasz-1}I_n=\sum_{j\in J}a_ju_j\otimes u_j\quad\hbox{and}\quad \sum_{j\in J}a_ju_j=0.\end{equation}
By the Dvoretzky-Rogers lemma, we may choose $n$ of these contact points, which we denote by $v_1,\ldots ,v_n$, so that
\begin{equation}\label{eq:nasz-2}{\rm dist}(v_k,{\rm span}(v_1,v_2,\ldots ,v_{k-1}))\gr\sqrt{\frac{n-k+1}{n}}\end{equation}
for all $k=2,\ldots ,n$. It follows that the simplex $S={\rm conv}\{ v_0=0,v_1,\ldots ,v_n\}\subseteq P$ has volume
\begin{equation}\label{eq:nasz-3}|S|= \frac{1}{n!}\prod_{k=1}^n{\rm dist}(v_k,{\rm span}(v_1,v_2,\ldots ,v_{k-1}))\gr
\frac{1}{n^{\frac{n}{2}}\sqrt{n!}}.\end{equation}
Now we use the fact (see \cite[Theorem 4.1.20]{AGA-book}) that if $w$ is the center of mass of $S$ then $S-w$ contains an origin symmetric
convex body $T_1$ of volume $|T_1|\gr 2^{-n}|S-w|=2^{-n}|S|$, and hence the convex body $T=T_1+w\subseteq S$ has a center of symmetry at
$w$ and satisfies
\begin{equation}\label{eq:nasz-4}|T|\gr 2^{-n}|S|.\end{equation}
Consider the ray $\ell $ from the origin in the direction of $-w$. Then, $\ell $ intersects the boundary of ${\rm conv}\{u_j,j\in J\}$
at a point $z\in {\rm conv}\{ v_{n+1},\ldots ,v_{n+k}\}$ for some $v_{n+i}\in\{u_j, j\in J\}$
and $k\ls n$ (this follows from Carath\'{e}odory's theorem). Also, note that ${\rm conv}\{u_j,j\in J\}\supseteq\frac{1}{n}B_2^n$, and hence
$\|z\|_2\gr\frac{1}{n}$. Applying a contraction with center $z$ and ratio
$$\lambda =\frac{\|z\|_2}{\|z-w\|_2}=\frac{\|z\|_2}{\|z\|_2+\|w\|_2}\gr \frac{\|z\|_2}{1+\|z\|_2}\gr\frac{1}{n+1}$$
to $T$, we obtain an origin symmetric convex body
\begin{equation}\label{eq:nasz-5}Q\subseteq {\rm conv}\{z,v_1,\ldots ,v_n\}\subseteq {\rm conv}\{v_1,\ldots ,v_n,v_{n+1},\ldots ,v_{n+k}\}\end{equation}
with volume
\begin{equation}\label{eq:nasz-6}|Q|\gr\frac{1}{(n+1)^n}|T|\gr \frac{1}{2^n(n+1)^n}|S|\gr\frac{1}{2^n(n+1)^nn^{\frac{n}{2}}\sqrt{n!}}.\end{equation}
Consider the intersection of $n+k\ls 2n$ half-spaces
\begin{equation}\label{eq:nasz-7}R=\bigcap_{i=1}^{n+k}\{x\in {\mathbb R}^n:\langle x,v_i\rangle\ls 1\}.\end{equation}
Using the Blaschke-Santal\'{o} inequality for $Q$ and the fact that $B_2^n\subseteq P$ and $R\subseteq Q^{\circ }$ we get
\begin{equation}\label{eq:nasz-8}\frac{|R|}{|P|}\ls \frac{|Q^{\circ }|}{|B_2^n|}\ls\frac{|B_2^n|}{|Q|}.\end{equation}
Finally, from \eqref{eq:nasz-6} we see that
\begin{equation}\label{eq:nasz-9}|R|\ls \frac{\pi^{\frac{n}{2}}2^n(n+1)^nn^{\frac{n}{2}}\sqrt{n!}}{\Gamma\left(\frac{n}{2}+1\right )}|P|\end{equation}
and the result follows (with constant $C=2\sqrt[3]{\pi }$ as one can check using Stirling's formula) . \end{proof}

\section{Approximate John's decompositions}

Our first main tool is the work of Batson, Spielman and Srivastava \cite{BSS-2009} on spectral sparcification
of graphs, in which they introduced a deterministic method extracting an
approximate John's decomposition starting from a John's decomposition of the form \eqref{eq:decomposition}.
Their result is the following:

\begin{theorem}[Batson-Spielman-Srivastava]\label{th:contact-srivastava}Let $v_1,\ldots ,v_m \in S^{n-1}$ and $a_1,\ldots ,a_m>0$ such that
\begin{equation}I_n=\sum_{j=1}^{m}a_jv_j\otimes v_j.\end{equation}
Then, for every $d>1$ there exists a subset $\sigma\subseteq \{ 1,\ldots ,m\}$ with $|\sigma |\ls dn$ and $b_j>0$, $j\in \sigma $, such that
\begin{equation}I_n\preceq \sum_{j\in\sigma }b_ja_jv_j\otimes v_j \preceq \gamma_dI_n,\end{equation}
where
\begin{equation}\gamma_d:=\left(\frac{\sqrt{d}+1}{\sqrt{d}-1}\right)^2.\end{equation}
\end{theorem}

Here, given two symmetric positive definite matrices $A$ and $B$ we write $A\preceq B$ if $\langle Ax,x\rangle \ls \langle Bx,x\rangle $
for all $x\in {\mathbb R}^n$. Using this fact, Srivastava \cite{Srivastava-2012} obtained an improved version of Rudelson's theorem \cite{Rudelson-1997}
on the approximation of a symmetric convex body $K$ by a symmetric convex body $T$ which has few contact points
with its maximal volume ellipsoid: for any symmetric convex body $K$ in ${\mathbb R}^n$ and any $\epsilon >0$ there exists a
 symmetric convex body $T$ such that $T\subseteq K\subseteq (1+\epsilon )T$
and $T$ has at most $32n/\epsilon^2$ contact points with its John ellipsoid.

In order to deal with the not-necessarily symmetric case of this question, Srivastava proved in \cite{Srivastava-2012}
the next variant of Theorem \ref{th:contact-srivastava}:

\begin{theorem}[Srivastava]\label{th:sriv}Let $v_1,\ldots ,v_m\in S^{n-1}$ and $a_1,\ldots ,a_m>0$ such that
\begin{equation}I_n=\sum_{j=1}^ma_jv_j\otimes v_j\quad\hbox{and}\quad \sum_{j=1}^ma_jv_j=0.\end{equation}
Given $\varepsilon>0$ we can find a subset $\sigma $ of $\{1,\ldots ,m\}$ of cardinality $|\sigma |=O_{\varepsilon}(n)$,
positive scalars $c_i$, $i\in \sigma $ and a vector $v$ with
\begin{equation}\label{eq:vector-v}\|v\|_2^2\ls\frac{\varepsilon}{\sum_{i\in \sigma } c_i},\end{equation}
such that
\begin{equation}\label{Srivastavalemma2}
I_n\preceq\sum_{i\in \sigma } c_i(v_i+v)\otimes (v_i+v)\preceq (4+\varepsilon)I_n
\end{equation}
and \begin{equation}\label{Srivastavalemma3}\sum_{i\in \sigma } c_i(v_i+v)=0.\end{equation}
\end{theorem}

Using Theorem \ref{th:sriv}, Srivastava showed that for any convex body $K$ in ${\mathbb R}^n$ and any $\epsilon >0$
there exists a convex body $T$ such that $T\subseteq K\subseteq (\sqrt{5}+\epsilon )T$
and $T$ has at most $O_{\epsilon }(n)$ contact points with its John ellipsoid. We will use Theorem \ref{th:sriv}
in order to deal with the not-necessarily symmetric case of our problem, which is clearly much more interesting
than the symmetric one.

\section{Brascamp-Lieb inequality and approximate John decompositions}

The Brascamp-Lieb inequality \cite{Brascamp-Lieb-1976} estimates the norm of the multilinear operator $G:L^{p_1}({\mathbb
R})\times\cdots\times L^{p_m}({\mathbb R})\to {\mathbb R}$ defined
by
\begin{equation}G(f_1,\ldots ,f_m)=\int_{{\mathbb R}^n}\prod_{j=1}^mf_j(\langle x,u_j\rangle )\;dx, \end{equation}
where $m\gr n$, $p_1,\ldots ,p_m\gr 1$ with $\frac{1}{p_1}+\cdots
+\frac{1}{p_m}=n$, and $u_1,\ldots ,u_m\in {\mathbb R}^n$. Brascamp
and Lieb proved that the norm of $G$ is the supremum $D$ of
\begin{equation}\frac{G(g_1,\ldots ,g_m)}{\prod_{j=1}^m\| g_j\|_{p_j}}\end{equation}
over all centered Gaussian functions $g_1,\ldots ,g_m$, i.e. over
all functions of the form $g_j(t)=e^{-\lambda_jt^2}$, $\lambda_j>0$.

If we set $c_j=1/p_j$ and replace $f_j$ by $f_j^{c_j}$ then we can
state the Brascamp-Lieb inequality in the following form.

\begin{theorem}[Bracamp-Lieb]\label{th:brascamp-lieb}
Let $m\gr n$, and let $u_1,\ldots ,u_m\in {\mathbb R}^n$ and
$c_1,\ldots ,c_m>0$ with $c_1+\cdots +c_m=n$. Then,
\begin{equation}\int_{{\mathbb R}^n}\prod_{j=1}^mf_j^{c_j}(\langle x,u_j\rangle )dx\ls D\, \prod_{j=1}^m\left (
\int_{{\mathbb R}}f_j\right )^{c_j}\end{equation} for all
integrable functions $f_j:{\mathbb R}\to [0,\infty )$, where $D=1/\sqrt{F}$ and
\begin{equation}F=\inf\left\{ \frac{\det \big (\sum_{j=1}^mc_j\lambda_ju_j\otimes u_j\big )}{\prod_{j=1}^m\lambda_j^{c_j}}
: \lambda_j>0\right\} .\end{equation}
\end{theorem}

Calculating the constant $F=F(\{ u_j\},\{ c_j\})$ in Theorem  \ref{th:brascamp-lieb} seems difficult. An
important observation of Ball (see e.g. \cite{Ball-1989}) is that if $u_1,\ldots ,u_m\in S^{n-1}$ and $c_1,\ldots ,c_m>0$ satisfy John's
decomposition of the identity \eqref{eq:decomposition} then the constant $F=F(\{ u_j\},\{ c_j\})$ in Theorem
\ref{th:brascamp-lieb} is equal to $1$.

\medskip

The next proposition shows that we still have a Brascamp-Lieb inequality with a reasonable constant when an approximate John's decomposition is available.

\begin{proposition}\label{deterlemma}Let $\gamma >1$. If $u_1,\ldots ,u_s \in S^{n-1}$ and $c_1,\ldots ,c_s>0$ satisfy
\begin{equation}I_n\preceq A:=\sum_{j=1}^sc_ju_j\otimes u_j \preceq \gamma I_n\end{equation}
then
\begin{equation}\gamma^n\det\left(\sum_{j=1}^s\kappa_j\lambda_j u_j\otimes u_j\right)\gr\prod_{j=1}^s\lambda_j^{\kappa_j}\end{equation}
for all $\lambda_1,\ldots ,\lambda_s>0$, where $\kappa_j=c_j\langle A^{-1}u_j, u_j\rangle >0$, $1\ls j\ls s$.
\end{proposition}

\begin{proof}For every $M\subset \{1,\ldots ,s\}$ with cardinality $|M|=n$ we define
\begin{equation}\lambda_M=\prod_{j\in M}\lambda_j\ \ \text{and}\ \ U_M=\det\left(\sum_{j\in M}c_ju_j\otimes u_j\right).\end{equation}
By the Cauchy-Binet formula we have
\begin{equation}\label{eq1}
\det\left(\sum_{j=1}^sc_j\lambda_j u_j\otimes u_j\right)=\sum_{|M|=n}\lambda_MU_M.
\end{equation}
Choosing $\lambda_j=1$ in \eqref{eq1} we get
\begin{equation}\sum_{|M|=n}U_M=\det(A).\end{equation}
By the arithmetic-geometric means inequality,
\begin{equation}\label{eq2}
\sum_{|M|=n}\lambda_M\frac{U_M}{\sum_{|M|=n} U_M}\gr\prod_{|M|=n}\lambda_M^{\frac{U_M}{\sum_{|M|=n} U_M}}
=\prod_{j=1}^s\lambda_j^{\frac{\sum_{\{M:j\in M\}}U_M}{\sum_{|M|=n} U_M}}.
\end{equation}
Applying the Cauchy-Binet formula again, we get
\begin{align*}
\frac{\sum_{\{M:j\in M\}}U_M}{\sum_{|M|=n} U_M}&=\frac{\sum_{|M|=n}U_M-\sum_{\{M:j\notin M\}}U_M}{\sum_{|M|=n} U_M}=1-\frac{\det\left(A-c_ju_j\otimes u_j\right)}{\det(A)}\\
&=1-(1-c_j\langle A^{-1}u_j,u_j\rangle )=c_j\langle A^{-1}u_j,u_j\rangle
\end{align*}
for every $j=1,\ldots ,s$, where in the last equality we used Lemma \ref{mdetlemma}. Going back to \eqref{eq1} and \eqref{eq2} we see that
\begin{equation}\label{eq3}
\frac{\det\left(\sum_{j=1}^sc_j\lambda_j u_j\otimes u_j\right)}{\det(A)}\gr\prod_{j=1}^s\lambda_j^{c_j\langle A^{-1}u_j,u_j\rangle }
\end{equation}
We set
\begin{equation}\kappa_j=c_j\langle A^{-1}u_j,u_j\rangle ,\qquad j=1,\ldots ,s.\end{equation}
Since $I_n\preceq A\preceq \gamma I_n$ we have that $\det(A)\gr 1$ and $\gamma\kappa_j=c_j\gamma \langle A^{-1}u_j,u_j\rangle \gr c_j$ for all $1\ls j\ls s$.
This implies that, for all $\lambda_1,\ldots ,\lambda_s>0$,
\begin{equation}\label{eq31}\sum_{j=1}^sc_j\lambda_j u_j\otimes u_j\preceq\gamma\left(\sum_{j=1}^s\kappa_j\lambda_j u_j\otimes u_j\right).\end{equation}
Combining \eqref{eq3} and \eqref{eq31} we get
\begin{equation}\gamma^n\det\left(\sum_{j=1}^s\kappa_j\lambda_j u_j\otimes u_j\right)\gr\prod_{j=1}^s\lambda_j^{\kappa_j}\end{equation}
as claimed. \end{proof}

\begin{remark}\rm Setting $\lambda_1=\cdots =\lambda_s=\lambda>0$ in the conclusion of Proposition \ref{deterlemma}, we get
\begin{equation}\gamma^n\lambda^n\det\left(\sum_{j=1}^s\kappa_ju_j\otimes u_j\right)\gr \lambda^{\sum_{j=1}^s\kappa_j}.\end{equation}
Since this holds true for any $\lambda >0$, we must have
\begin{equation}\label{eq:cond}\sum_{j=1}^s\kappa_j=n.\end{equation}
We can also check this directly: note that
\begin{align}\sum_{j=1}^s\kappa_j &= \sum_{j=1}^sc_j\langle A^{-1}u_j,u_j\rangle =\sum_{j=1}^sc_j\,{\rm tr}(u_j\otimes A^{-1}u_j)={\rm tr}\left (\sum_{j=1}^sc_j(u_j\otimes A^{-1}u_j)\right )\\
\nonumber &={\rm tr}\left (\sum_{j=1}^sc_j A^{-1}(u_j\otimes u_j)\right )= {\rm tr}\left (A^{-1}\Big (\sum_{j=1}^sc_j(u_j\otimes u_j)\Big )\right )
={\rm tr}(A^{-1}A)={\rm tr}(I_n)=n.
\end{align}
Having verified condition \eqref{eq:cond}, we conclude from Proposition \ref{deterlemma} that the constant in the Brascamp-Lieb inequality that corresponds to $\{u_j\}_{j=1}^s$ and $\{ \kappa_j\}_{j=1}^s$ is bounded by $\gamma^{n/2}$.
We will use this observation in the following form:
\end{remark}

\begin{theorem}\label{BL-approx}Let $\gamma >1$. Let $u_1,\ldots ,u_s \in S^{n-1}$ and $c_1,\ldots ,c_s>0$ satisfy
\begin{equation}I_n\preceq A:=\sum_{j=1}^sc_ju_j\otimes u_j \preceq \gamma I_n\end{equation}
and set $\kappa_j=c_j\langle A^{-1}u_j, u_j\rangle >0$, $1\ls j\ls s$. If
$f_1,\ldots ,f_s:{\mathbb R}\to {\mathbb R}^+$ are integrable
functions then
\begin{equation}\int_{{\mathbb R}^n}\prod_{j=1}^sf_j^{\kappa_j}(\langle x,u_j\rangle )dx
\ls \gamma^{\frac{n}{2}}\prod_{j=1}^s \left(\int_{{\mathbb R}} f_j(t)dt\right)^{\kappa_j}.\end{equation}
\end{theorem}

\section{Volume approximation by convex bodies with few facets}

In this section we prove the main theorems of this article. We show that the intersection of any
family of closed half-spaces is contained in an intersection of $N\simeq n$ of these
half-spaces whose volume is reasonably small. This implies our quantitative versions of Helly's theorem as explained in the
introduction.

We start with the symmetric case.

\begin{theorem}\label{th:strips}Let $\{P_i:i\in I\}$ be a family of symmetric strips
\begin{equation}P_i=\{ x\in {\mathbb R}^n:|\langle x,v_i\rangle |\ls 1\}\end{equation}
in ${\mathbb R}^n$, and let $P=\bigcap_{i\in I}P_i$. For every $d>1$ there exist $s\ls dn$ and $i_1,\ldots ,i_s\in I$
such that
\begin{equation}|P_{i_1}\cap\cdots\cap P_{i_s}|\ls \left (\frac{2}{\sqrt{\pi}}\frac{\sqrt{d}+1}{\sqrt{d}-1}\right)^n\Gamma \left(\frac{n}{2}+1\right )|P|.\end{equation}
\end{theorem}

\begin{proof}We may assume that $P$ is in John's position. From John's theorem there exists $J\subseteq I$ so that the vectors $v_j$, $j\in J$
are contact points of $P$ and $S^{n-1}$ and there exist $a_j>0$, $j\in J$, such that
\begin{equation}I_n=\sum_{j\in J}a_jv_j\otimes v_j.\end{equation}
Theorem \ref{th:contact-srivastava} shows that there exists a subset $\sigma\subseteq J$ with $|\sigma |=s\ls dn$ and $b_j>0$, $j\in \sigma $, such that
\begin{equation}I_n\preceq \sum_{j\in\sigma }b_ja_jv_j\otimes v_j \preceq \gamma_dI_n,\end{equation}
where $\gamma_d=\left(\frac{\sqrt{d}+1}{\sqrt{d}-1}\right)^2$. We rewrite the vectors $v_j$, $j\in \sigma $, as $w_1,\ldots ,w_s$ and
we set $c_j=a_jb_j$. Now, we apply Theorem \ref{BL-approx} to find $\kappa_j>0$, $1\ls j\ls s$ such that $\sum_{j=1}^s\kappa_j=n$ and
\begin{equation}\int_{{\mathbb R}^n}\prod_{j=1}^sf_j^{\kappa_j}(\langle x,w_j\rangle )dx
\ls \gamma_d^{\frac{n}{2}}\prod_{j=1}^s \left(\int_{{\mathbb R}} f_j(t)dt\right)^{\kappa_j}\end{equation}
for any choice of non-negative integrable functions $f_1,\ldots ,f_s$ on ${\mathbb R}^n$. Note that
\begin{equation}|P_1\cap\cdots\cap P_s| = \int_{{\mathbb R}^n} \prod_{j=1}^s {\bf
1}_{[-1,1]}(\langle x,w_j\rangle )^{\kappa_j}dx .\end{equation}
Since $\int_{{\mathbb R}}{\bf 1}_{[-1,1]}(t)dt=2$, from Theorem \ref{BL-approx} we get
\begin{equation}|P_1\cap\cdots\cap P_s| \ls 2^n\gamma_d^{\frac{n}{2}}.\end{equation}
Since $B_2^n\subseteq P$, we also have
\begin{equation}|P|\gr |B_2^n|= \frac{\pi^{n/2}}{\Gamma \left (\frac{n}{2}+1\right )}\end{equation}
and the result follows.
\end{proof}

\begin{remark}\rm The proof of Theorem \ref{th:strips} shows that if $K$ is a symmetric convex body in
John's position then for every $d>1$ there exist $s\ls dn$ and $w_1,\ldots ,w_s\in S^{n-1}$ such that
\begin{equation}K\subseteq P:=\bigcap_{j=1}^s\{ x\in {\mathbb R}^n:|\langle x,w_j\rangle |\ls 1\}\end{equation}
and
\begin{equation}\label{eq:upper-1}|P|^{\frac{1}{n}}\ls 2\frac{\sqrt{d}+1}{\sqrt{d}-1}.\end{equation}
This estimate should be compared to well-known lower bounds for the volume of intersections of strips, due
to Carl-Pajor \cite{Carl-Pajor-1988}, Gluskin \cite{Gluskin-1988} and Ball-Pajor \cite{Ball-Pajor-1990}.
If we fix $d>1$ and set $N=\lfloor dn\rfloor $ then for any choice of vectors $w_1,\ldots ,w_N$ spanning ${\mathbb R}^n$, with $\|
w_i\|_2\ls 1$ for all $1\ls i\ls N$, we know that the body $P=\bigcap_{j=1}^N\{ x\in {\mathbb R}^n:|\langle x,w_j\rangle |\ls 1\}$ satisfies
\begin{equation}\label{lower}|P|^{\frac{1}{n}}\gr \frac{2}{\sqrt{e}\sqrt{\log (1+d)}}.\end{equation}
which is of the same order (up to the dependence on $d$).

On the other hand, even if we ask that $N=n$ (which corresponds to $d=1$), one may find upper estimates
of the form \eqref{eq:upper-1} in the literature: for example, if $K$ is a symmetric convex body in John's position and if $v_1,\ldots ,v_n$ are
the vectors in \eqref{eq:1.1} then the parallelepiped
\begin{equation}\label{eq:1.2}P=\{ x\in {\mathbb R}^n:|\langle x,v_j\rangle |\ls 1,\,j=1,\ldots ,n\},\end{equation}
satisfies $K\subseteq P$ and
\begin{equation}\label{eq:1.3}|P|^{\frac{1}{n}}=2|\det (v_1,v_2,\ldots ,v_n)|^{-\frac{1}{n}}\ls \frac{2\sqrt{n}}{(n!)^{\frac{1}{2n}}}\sim
2\sqrt{e}.\end{equation}
This result is due to Dvoretzky and Rogers, and an estimate of the same order (but improving in a sense the constants involved)
was obtained by Pelczynski and Szarek in \cite{Pelczynski-Szarek-1991}. Comparing $2\sqrt{e}$ with $2\frac{\sqrt{d}+1}{\sqrt{d}-1}$
we see that our estimate provides a better bound if we allow a larger, but still proportional to the dimension, number of strips.
\end{remark}

\medskip

Next, we pass to the not-necessarily symmetric case; we consider a family $\{P_i:i\in I\}$ of closed half-spaces and ask
for a collection of $s$ half-spaces $P_j$ such that $|P_1\cap\cdots \cap P_s|\ls c_{n,s}\left |\bigcap_{i\in I}P_i\right |$.
We give two arguments. The first one is based on the ideas of Theorem \ref{th:strips} and establishes (for any $d>1$) a choice of $s\ls (d+1)(n+1)$ half-spaces and a bound of the order of $n^{3n/2}$ for the constant $c_{n,s}$.

\begin{theorem}\label{th:halfspaces}Let $\{P_i:i\in I\}$ be a family of closed half-spaces
\begin{equation}P_i=\{ x\in {\mathbb R}^n:\langle x,u_i\rangle \ls 1\}\end{equation}
in ${\mathbb R}^n$, such that $P=\bigcap_{i\in I}P_i$ has positive volume. For every $d>1$ there exist $s\ls (d+1)(n+1)$ and
$i_1,\ldots , i_s\in I$ such that
\begin{equation}|P_{i_1}\cap\cdots\cap P_{i_s}|\ls \gamma_d^{\frac{n+1}{2}}\frac{n^{n/2}(n+1)^{3(n+1)/2}}{\pi^{\frac{n}{2}}n!}\Gamma \left(\frac{n}{2}+1\right )\,|P|
\ls \gamma_d^{\frac{n+1}{2}}(Cn)^{\frac{3n}{2}}|P|,\end{equation}where $C>0$ is an absolute constant.
\end{theorem}

\begin{proof}We may assume that $P$ is in John's position. From John's theorem there exists $J\subseteq I$ so that the vectors $u_j$, $j\in J$
are contact points of $P$ and $S^{n-1}$ and there exist $a_j>0$, $j\in J$, such that
\begin{equation}I_n=\sum_{j\in J}a_ju_j\otimes u_j\quad\hbox{and}\quad \sum_{j\in J}a_ju_j=0.\end{equation}
Set \begin{equation}v_j=\sqrt{\frac{n}{n+1}}\left(-u_j,\frac{1}{\sqrt{n}}\right)\ \ \text{and}\ \ b_j=\frac{n+1}{n}a_j.\end{equation}
Then \begin{equation}I_{n+1}=\sum_{j\in J}b_jv_j\otimes v_j.\end{equation}
Theorem \ref{th:contact-srivastava} shows that there exists a subset $\sigma\subseteq J$ with $|\sigma |=s\ls d(n+1)$ and $\delta_j>0$, $j\in \sigma $, such that
\begin{equation}I_{n+1}\preceq A:=\sum_{j\in\sigma }\delta_jb_jv_j\otimes v_j \preceq \gamma_dI_{n+1},\end{equation}
where $\gamma_d=\left(\frac{\sqrt{d}+1}{\sqrt{d}-1}\right)^2$. We fix the vectors $v_j$, $j\in \sigma $, and set $c_j=\delta_jb_j$.
We also consider the vector
\begin{equation}w:=-\frac{1}{n(n+1)}\sum_{j\in\sigma }\kappa_ju_j,\end{equation}
where $\kappa_j=c_j\langle A^{-1}u_j, u_j\rangle >0$, $j\in\sigma $ are the scalars provided by Proposition \ref{deterlemma}.
Recall that, by John's theorem, ${\rm conv}\{u_j,j\in J\}\supseteq \frac{1}{n}B_2^n$, and $\|w\|_2\ls \frac{1}{n}$ by the triangle inequality and the fact that $\sum_{j\in \sigma }\kappa_j= n+1$. From Carath\'{e}odory's theorem we get that there exists $\tau\subseteq J$
 with $|\tau |\ls n+1$ and $\rho_i> 0$ with $\sum_{i\in\tau }\rho_i=1$ so that
\begin{equation}w=\sum_{i\in\tau }\rho_iu_i.\end{equation}
We define
\begin{equation}Q=\{ x\in {\mathbb R}^n:\langle x,u_j\rangle < 1\;\;\hbox{for all}\;j\in\sigma\}\end{equation}
and
\begin{equation}Q^{\prime }=Q\cap \{x\in {\mathbb R}^n:\langle x,u_i\rangle\ls 1\;\;\hbox{for all}\;i\in\tau \}.\end{equation}
From Theorem \ref{BL-approx} we know that if
$f_j:{\mathbb R}\to {\mathbb R}^+$, $j\in\sigma $ are integrable functions, then
\begin{equation}\label{eq:BL-hs-1}\int_{{\mathbb R}^{n+1}}\prod_{j\in\sigma }f_j^{\kappa_j}(\langle y,v_j\rangle )dy
\ls \gamma_d^{\frac{n+1}{2}}\prod_{j\in\sigma }\left(\int_{{\mathbb R}} f_j(t)dt\right)^{\kappa_j}.\end{equation}
For $j\in\sigma $ we define $f_j(t)=e^{-t}{\bf 1}_{[0,\infty )}(t)$. Let $y=(x,r)\in {\mathbb R}^{n+1}$. We easily check that if
$r>0$ and $x\in\frac{r}{\sqrt{n}}Q$ then $\langle x,u_j\rangle <\frac{r}{\sqrt{n}}$ for all $j\in\sigma $. This implies
that $\langle y,v_j\rangle >0$ for all $j\in\sigma $, and hence
$\prod_{j\in\sigma }f_j^{\kappa_j}(\langle y,v_j\rangle )>0$. It follows that
\begin{align}
\int_{{\mathbb R}^{n+1}}\prod_{j\in\sigma }f_j^{\kappa_j}(\langle y,v_j\rangle )dy
&\gr \int_0^{\infty }\int_{\frac{r}{\sqrt{n}}Q}\prod_{j\in\sigma }f_j^{\kappa_j}(\langle y,v_j\rangle )dy\\
\nonumber &= \int_0^{\infty }\int_{\frac{r}{\sqrt{n}}Q}\exp\left (-\sum_{j\in\sigma }\kappa_j\langle (x,r),v_j\rangle \right )dx\,dr\\
\nonumber &= \int_0^{\infty }\int_{\frac{r}{\sqrt{n}}Q}\exp\left (\sqrt{\frac{n}{n+1}}\sum_{j\in\sigma }\kappa_j\langle x,u_j\rangle -\frac{r}{\sqrt{n+1}}\sum_{j\in\sigma }\kappa_j\right )dx\,dr\\
\nonumber &= \int_0^{\infty }\int_{\frac{r}{\sqrt{n}}Q}e^{-r\sqrt{n+1}}\exp\left (-n^{3/2}\sqrt{n+1}\langle x,w\rangle \right )dx\,dr\\
\nonumber &\gr \int_0^{\infty }\int_{\frac{r}{\sqrt{n}}Q^{\prime }}e^{-r\sqrt{n+1}}\exp\left (-n^{3/2}\sqrt{n+1}\,\langle x,w\rangle \right )dx\,dr,
\end{align}
where, in the last step, we use the fact that $Q^{\prime }\subseteq Q$. Now, observe that if $x\in \frac{r}{\sqrt{n}}Q^{\prime }$ then
\begin{equation}\langle x,w\rangle =\sum_{i\in\tau }\rho_i\langle x,u_i\rangle \ls \frac{r}{\sqrt{n}}.\end{equation}
So, we get
\begin{align}
\int_{{\mathbb R}^{n+1}}\prod_{j\in\sigma }f_j^{\kappa_j}(\langle y,v_j\rangle )dy &\gr
\int_0^{\infty }\int_{\frac{r}{\sqrt{n}}Q^{\prime }}e^{-r\sqrt{n+1}-rn\sqrt{n+1}}dx\,dr\\
\nonumber &= \int_0^{\infty }\int_{\frac{r}{\sqrt{n}}Q^{\prime }}e^{-r(n+1)^{3/2}}dx\,dr\\
\nonumber &= |Q^{\prime }|\cdot\frac{1}{n^{n/2}}\int_0^{\infty }r^ne^{-r(n+1)^{3/2}}\,dr\\
\nonumber &= |Q^{\prime }|\cdot\frac{1}{n^{n/2}}\frac{n!}{(n+1)^{3(n+1)/2}}.
\end{align}Note that
\begin{equation}\prod_{j\in\sigma } \left(\int_{{\mathbb R}} f_j(t)dt\right)^{\kappa_j}=1,\end{equation}
and hence \eqref{eq:BL-hs-1} gives us
\begin{equation}|Q^{\prime }|\ls \gamma_d^{\frac{n+1}{2}}\frac{n^{n/2}(n+1)^{3(n+1)/2}}{n!}.\end{equation}
Since $Q^{\prime }$ is an intersection of at most $(d+1)(n+1)$ half-spaces and $B_2^n\subseteq P\subseteq Q^{\prime }$, the result follows
as in the symmetric case. Using Stirling's formula one can check that the statement holds true with $C_d=\left(\frac{e\gamma_d}{2\pi }\right )^{\frac{1}{3}}$.
\end{proof}

\medskip

Our next argument provides (for an absolute constant $\alpha\gg 1$) a choice of $s\ls \alpha n$ half-spaces and a much better bound of the order of $n^{n}$ for the constant $c_{n,s}$.

\begin{theorem}\label{th:halfspaces2}There exists an absolute constant $\alpha >1$ with the following
property: for every family $\{P_i:i\in I\}$ of closed half-spaces
\begin{equation}P_i=\{ x\in {\mathbb R}^n:\langle x,u_i\rangle \ls 1\}\end{equation}
in ${\mathbb R}^n$, such that $P=\bigcap_{i\in I}P_i$ has positive volume, there exist $s\ls \alpha n$ and
$i_1,\ldots , i_s\in I$ such that
\begin{equation}|P_{i_1}\cap\cdots\cap P_{i_s}|\ls (Cn)^n\,|P|,\end{equation}
where $C>0$ is an absolute constant.
\end{theorem}

\begin{proof}As in the proof of Theorem \ref{th:halfspaces} we assume that $P$ is in John's position, and we find $J\subseteq I$ so that the vectors $u_j$, $j\in J$
are contact points of $P$ and $S^{n-1}$ and there exist $a_j>0$, $j\in J$, such that
\begin{equation}I_n=\sum_{j\in J}a_ju_j\otimes u_j\quad\hbox{and}\quad \sum_{j\in J}a_ju_j=0.\end{equation}
We apply Theorem \ref{th:sriv} to find a subset $\sigma \subseteq J$ with $|\sigma |\ls\alpha_1(\varepsilon )n$, positive scalars $c_j$, $j\in \sigma $ and a vector $u$ such that
\begin{equation}\label{half11}
I_n\preceq\sum_{j\in \sigma} c_j(u_j+u)\otimes (u_j+u)\preceq (4+\varepsilon )I_n
\end{equation}
and \begin{equation}\label{half12}\sum_{j\in \sigma } c_j(u_j+u)=0\ \ \text{and}\ \ \|u\|_2^2\ls\frac{\varepsilon }{\sum_{j\in \sigma } c_j}.\end{equation}
Note that
\begin{align}
{\rm tr}\left(\sum_{j\in \sigma} c_j(u_j+u)\otimes (u_j+u)\right) &=\sum_{j\in\sigma }c_j\|u_j+u\|_2^2\\
\nonumber &=\sum_{j\in \sigma }c_j\|u_j\|_2^2+2\sum_{j\in\sigma }\langle u,c_ju_j\rangle+\left(\sum_{j\in \sigma }c_j\right)\|u\|_2^2 \\
\nonumber &=\sum_{j\in \sigma }c_j+2\Big\langle u,-\Big(\sum_{j\in \sigma }c_j\Big)u\Big\rangle+ \left (\sum_{j\in\sigma }c_j\right )\|u\|^2_2\\
\nonumber &=\sum_{j\in \sigma }c_j-\left (\sum_{j\in\sigma }c_j\right )\|u\|_2^2
\end{align}
and hence from \eqref{half11} we get that
$$n\ls\sum_{j\in \sigma }c_j-\left (\sum_{j\in\sigma }c_j\right )\|u\|_2^2\ls (4+\varepsilon )n.$$
Now, using \eqref{half12} we get
\begin{equation}\label{ineq1}
n\ls\sum_{j\in \sigma}c_j\ls (4+2\varepsilon )n.
\end{equation}
In particular,
\begin{equation}\|u\|_2^2\ls\frac{\varepsilon }{\sum_{j\in\sigma }c_j}\ls\frac{\varepsilon }{n}.\end{equation}
Recall that $\conv\{u_j,j\in J\}\supseteq\frac{1}{n}B_2^n$. Then, for the vector $w=\frac{u}{\sqrt{\varepsilon n}}$ we have $\|w\|_2\ls\frac{1}{n}$ and hence $w\in \conv\{u_j, j\in J\}$. Carath\'{e}odory's theorem shows that there exist $\tau\subseteq J$ with $|\tau |\ls n+1$ and $\rho_i> 0$, $i\in\tau $ such that
\begin{equation}w=\sum_{i\in\tau }\rho_iu_i\ \ \text{and}\ \ \sum_{i\in\tau }\rho_i=1.\end{equation}
Note that \begin{equation}\left (\sum_{j\in \sigma}c_j\right )(-u)=\sum_{j\in \sigma }c_ju_j,\end{equation}
and this shows that $-u\in {\rm conv}\{u_j:j\in \sigma\}$. It follows that the segment
 \begin{equation}\left[-u,\frac{u}{\sqrt{\varepsilon n}}\right]\subset {\rm conv}\{u_j:j\in \sigma\cup\tau \}.\end{equation}
For $j\in\sigma $ we set \begin{equation}v_j=\sqrt{\frac{n}{n+1}}\left(-u_j,\frac{1}{\sqrt{n}}\right)\ \ \text{and}\ \ b_j=\frac{n+1}{n}c_j.\end{equation}We also set $-v=\sqrt{\frac{n}{n+1}}( u,0)$.
Then, using \eqref{half12} we get
\begin{equation}
\sum_{j\in\sigma} b_j(v_j+v)\otimes (v_j+v)=
\begin{pmatrix}\sum_{j\in\sigma} c_j(u_j+u)\otimes (u_j+u) & 0 \\
0 & \frac{\sum_{j\in\sigma} c_j}{n}
\end{pmatrix},
\end{equation}
which implies, with the help of \eqref{ineq1}, that
\begin{equation}
I_{n+1}\preceq\sum_{j\in \sigma} b_j(v_j+v)\otimes (v_j+v)\preceq (4+2\varepsilon )I_{n+1}.
\end{equation}
We rewrite the last one as follows:
\begin{align}\label{eq:6-44}
& I_{n+1}-\sum_{j\in \sigma} b_jv_j\otimes v-\sum_{j\in \sigma} v\otimes b_jv_j-\left(\sum_{j\in \sigma} b_j\right)v\otimes v\\
\nonumber &\hspace*{2cm}\preceq\sum_{j\in \sigma} b_jv_j\otimes v_j\preceq 5I_{n+1}-\sum_{j\in \sigma} b_jv_j\otimes v-\sum_{j\in \sigma} v\otimes b_jv_j-\left(\sum_{j\in \sigma} b_j\right)v\otimes v.
\end{align}
Note that \begin{equation}
\sum_{j\in \sigma} b_jv_j=\sqrt{\frac{n+1}{n}}\left(-\sum_{j\in \sigma}c_ju_j,\frac{\sum_{j\in \sigma}c_j}{\sqrt{n}}\right)
=\sqrt{\frac{n+1}{n}}\left(\left(\sum_{j\in \sigma}c_j\right)u,\frac{\sum_{j\in \sigma}c_j}{\sqrt{n}}\right),
\end{equation}
so
\begin{align}
\left(\sum_{j\in \sigma} b_jv_j\right)\otimes v&=\left(\left(\sum_{j\in \sigma}c_j\right)u,\frac{\sum_{j\in \sigma}c_j}{\sqrt{n}}\right)\otimes (-u,0)\\
\nonumber &=\begin{pmatrix}
-\left(\sum_{j\in \sigma}c_j\right)u\otimes u & 0 \\
 -\frac{\left(\sum_{j\in \sigma}c_j\right)u}{\sqrt{n}} & 0
\end{pmatrix}.
\end{align}
Computing in a similar way we finally have that
\begin{equation}
T:=\sum_{j\in \sigma} b_jv_j\otimes v+\sum_{j\in \sigma} v\otimes b_jv_j+\left(\sum_{j\in \sigma} b_j\right)v\otimes v=\begin{pmatrix}
V & z \\
z & 0\end{pmatrix}.
\end{equation}
where $V=-\left(\sum_{j\in \sigma}c_j\right)u\otimes u$ and $z=-\frac{\left(\sum_{j\in \sigma}c_j\right)u}{\sqrt{n}}$. Now, for every $(x,t)\in S^n$ we have
\begin{align}
\langle T(x,t), (x,t)\rangle &=\langle Vx,x\rangle +2\langle z,t\rangle \ls \|V\|+2\|z\|_2\\
\nonumber &=\left (\sum_{j\in\sigma }c_j\right )\|u\|_2^2+\left (\sum_{j\in\sigma }c_j\right )\frac{2\|u\|_2}{\sqrt{n}}\\
\nonumber &\ls \varepsilon + (4+2\varepsilon )n\frac{2\sqrt{\varepsilon }}{n}= \varepsilon +2\sqrt{\varepsilon }(4+2\varepsilon ).
\end{align}
Choosing $\varepsilon =10^{-3}$ we get
\begin{equation}
\left\|\sum_{j\in \sigma} b_jv_j\otimes v+\sum_{j\in \sigma} v\otimes b_jv_j+\left(\sum_{j\in \sigma}b_j\right)v\otimes v\right\|
\ls\frac{1}{2},
\end{equation}
and going back to \eqref{eq:6-44} we get
\begin{equation}\label{approx1}
\frac{1}{2}I_{n+1}\preceq\sum_{j\in \sigma} b_jv_j\otimes v_j\preceq 5I_{n+1}.
\end{equation}
Now, we apply Proposition \ref{deterlemma} to find $\kappa_j>0$, $j\in \sigma $ such that if
$f_j:{\mathbb R}\to {\mathbb R}^+$ are measurable functions, then
\begin{equation}\int_{{\mathbb R}^{n+1}}\prod_{j\in \sigma }f_j^{\kappa_j}(\langle y,v_j\rangle )dy
\ls 10^{\frac{n+1}{2}}\prod_{j\in \sigma } \left(\int_{{\mathbb R}} f_j(t)dt\right)^{\kappa_j}.\end{equation}
For $j\in \sigma $ we define $f_j(t)=e^{-\frac{b_j}{k_j}t}{\bf 1}_{[0,\infty )}(t)$. Then,
\begin{equation}\label{BLresult}
\int_{{\mathbb R}^{n+1}}\prod_{j\in \sigma }f_j^{\kappa_j}(\langle y,v_j\rangle )dy
\ls 10^{\frac{n+1}{2}}\prod_{j\in\sigma } \left(\int_{{\mathbb R}} f_j(t)dt\right)^{\kappa_j}=10^{\frac{n+1}{2}}\prod_{j\in \sigma }\left(\frac{\kappa_j}{c_j}\right)^{\kappa_j}\ls 40^{\frac{n+1}{2}},
\end{equation}
recalling from the proof of Proposition \ref{deterlemma} that $\frac{\kappa_j}{b_j}=\langle A^{-1}u_j,u_j\rangle \ls 2$ (the last inequality is a
consequence of $\frac{1}{2}I_{n+1}\preceq A=\sum_{j\in \sigma }b_jv_j\otimes v_j$).

Let \begin{equation}Q=\{x\in {\mathbb R}^n:\langle x,u_j\rangle < 1,\ \ j\in \sigma \cup\tau \}.\end{equation}
We write $y=(x,r)\in {\mathbb R}^{n+1}$ and assume that
$r>0$ and $x\in\frac{r}{\sqrt{n}}Q$. Then, we have $\langle x,u_j\rangle<\frac{r}{\sqrt{n}}$ for all $j\in\sigma $. This implies that
$\langle y,v_j\rangle >0$ for all $j\in\sigma $, and hence
$\prod_{j\in\sigma }f_j^{\kappa_j}(\langle y,v_j\rangle )>0$. We also have
\begin{align}
\frac{1}{\left(\sum_{j\in \sigma } c_j\right)}\left\langle \sum_{j\in \sigma } c_ju_j,x\right\rangle&=\langle -u,x\rangle
=\sqrt{\varepsilon n}\langle -w,x\rangle =\sqrt{\varepsilon n}\left\langle -\sum_{i\in\tau } \rho_iu_i,x\right\rangle\\
\nonumber &\gr -\sqrt{\varepsilon }r,
\end{align}
where the last inequality holds since $x\in\frac{r}{\sqrt{n}}Q.$
It follows that
\begin{equation}\label{ineqproduct}
\Big\langle \sum_{j\in \sigma } c_ju_j,x\Big\rangle\gr -5\sqrt{\varepsilon }rn.
\end{equation}
Using the above (and recalling our choice of $\varepsilon =10^{-3}<1$) we see that if $y=(x,r)\in \frac{r}{\sqrt{n}}Q\times (0,\infty )$ then
\begin{align}
\prod_{j\in \sigma }f_j^{\kappa_j}(\langle y,v_j\rangle )&=\exp\left(-\sum_{j\in \sigma }b_j\left(\frac{r}{\sqrt{n}}-\sqrt{\frac{n}{n+1}}\langle x,u_j\rangle\right)\right)\\
\nonumber &=\exp \left (-\frac{r}{\sqrt{n}}\sum_{j\in \sigma }b_j\right )\exp\left(\Big\langle x,\sum_{j\in \sigma }b_ju_j\Big\rangle\right)\\
\nonumber &\gr \exp\left(-5r\frac{n+1}{\sqrt{n}}-5\sqrt{\varepsilon }r(n+1)\right)\gr\exp\left(-10r(n+1)\right).
\end{align}
Now, \eqref{BLresult} gives us
\begin{align}\frac{|Q|}{n^{\frac{n}{2}}}\int_{0}^{\infty} r^n e^{-10r(n+1)}\, dr &=
\int_0^{\infty }\int_{\frac{r}{\sqrt{n}}Q}e^{-10r(n+1)}dx\,dr
\ls \int_{{\mathbb R}^{n+1}}\prod_{j\in \sigma }f_j^{\kappa_j}(\langle y,v_j\rangle )dy\\
\nonumber &\ls 40^{\frac{n+1}{2}} .\end{align}
Direct computation and then Stirling's approximation show that
\begin{equation}|Q|\ls C_1^n\frac{n^{\frac{3n}{2}}}{n!}\ls C_2^n n^{\frac{n}{2}}\end{equation}
and $Q$ is the intersection of at most $|\sigma |+|\tau |\ls \alpha_1(10^{-3})n+n+1\ls \alpha n$ half-spaces, where $\alpha =\alpha_1(10^{-3})+2$.
Since $B_2^n\subseteq P\subseteq Q$, the result follows as in the symmetric case.
\end{proof}

\bigskip

\noindent {\bf Acknowledgement.} We would like to thank Apostolos Giannopoulos for suggesting the problem and for useful discussions.
We acknowledge support from the programme ``API$\Sigma$TEIA II -- ATOCB -- 3566" of the General Secretariat for Research and Technology of Greece.

\bigskip

\bigskip

\footnotesize
\bibliographystyle{amsplain}

\medskip

\thanks{\noindent {\bf Keywords:}  Convex bodies, Helly's theorem, John's decomposition, Brascamp-Lieb inequality.}

\smallskip

\thanks{\noindent {\bf 2010 MSC:} Primary 26D15; Secondary 52A23, 52A35, 46B06.}

\bigskip

\bigskip

\noindent \textsc{Silouanos \ Brazitikos}: Department of
Mathematics, University of Athens, Panepistimioupolis 157-84,
Athens, Greece.

\smallskip

\noindent \textit{E-mail:} \texttt{silouanb@math.uoa.gr}

\end{document}